\documentclass[11pt,reqno]{amsart}

\usepackage{graphicx}

\usepackage{xspace}
\usepackage{amsmath,amsthm,amsfonts,amssymb,latexsym,mathrsfs,color,extarrows}
\usepackage{hyperref}

\newtheorem{theorem}{Theorem}
\newtheorem{corollary}[theorem]{Corollary}
\newtheorem{proposition}[theorem]{Proposition}

\newtheorem{lemma}[theorem]{Lemma}
\newtheorem{definition}[theorem]{Definition}
\newtheorem{example}[theorem]{Example}

\newcommand{\waexc}{{\rm waexc\,}}
\newcommand{\ston}{{\rm st\,}}
\newcommand{\ST}{{\rm ST\,}}

\newcommand{\gen}{{\rm Gen\,}}

\newcommand{\basc}{{\rm basc\,}}

\newcommand{\st}{{\rm st\,}}

\newcommand{\cda}{{\rm cda\,}}
\newcommand{\CDD}{{\rm CDD\,}}

\newcommand{\suc}{{\rm suc\,}}

\newcommand{\lpk}{{\rm lpk\,}}

\newcommand{\des}{{\rm des\,}}

\newcommand{\exc}{{\rm exc\,}}

\newcommand{\aexc}{{\rm aexc\,}}
\newcommand{\we}{{\rm wexc\,}}

\newcommand{\cyc}{{\rm cyc\,}}

\newcommand{\fix}{{\rm fix\,}}

\newcommand{\NS}{\mathcal{NS}}
\newcommand{\mc}{{\mathcal C}}

\newcommand{\mdn}{\mathcal{D}}

\newcommand{\D}{\mathfrak{D}}

\newcommand{\msn}{\mathfrak{S}_n}
\newcommand{\rss}{\mathcal{SS}}
\newcommand{\ms}{\mathfrak{S}}

\newcommand{\rs}{\mathcal{RS}}

\newcommand{\lrf}[1]{\lfloor #1\rfloor}

\newcommand{\z}{ \mathbb{Z}}
\newcommand{\asc}{{\rm asc\,}}

\newcommand{\arxiv}[1]{\href{http://arxiv.org/abs/#1}{\texttt{arXiv:#1}}}
\title{ The $\gamma$-positivity of Eulerian polynomials and succession statistics}
\author[S.-M.~Ma]{Shi-Mei Ma}
\address{School of Mathematics and Statistics,
        Northeastern University at Qinhuangdao,
         Hebei 066000, P.R. China}
\email{shimeimapapers@163.com (S.-M. Ma)}
\author{Jun Ma}
\address{Department of Mathematics, Shanghai Jiao Tong University, Shanghai, P.R. China}
\email{majun904@sjtu.edu.cn}
\author{Jean Yeh}
\address{Departemnt of Mathematics, National Kaohsiung Normal University, Kaohsiung 82446, Taiwan}
\email{chunchenyeh@nknu.edu.tw}
\author{Yeong-Nan Yeh}
\address{Institute of Mathematics, Academia Sinica, Taipei, Taiwan}
\email{mayeh@math.sinica.edu.tw}
\begin{document}

\maketitle
\begin{abstract}
This paper is concerned with multivariate refinements of the $\gamma$-positivity of Eulerian polynomials by using the succession and fixed point statistics.
Properties of the enumerative polynomials for permutations, signed permutations and derangements, including generating functions and $\gamma$-positivity are
studied, which generalize and unify earlier results of Athanasiadis, Brenti, Chow, Petersen, Roselle, Stembridge, Shin and Zeng.
In particular, we derive a formula expressing the joint distribution of excedance number
and negative number statistics over the type $B$ derangements in terms of the derangement polynomials.
\bigskip

\end{abstract}
\date{\today}
\section{Introduction}
Let $f(x)=\sum_{i=0}^nf_ix^i$ be a symmetric polynomial of degree $n$, i.e., $f_i=f_{n-i}$ for any $0\leq i\leq n$. Then $f(x)$ can be expanded uniquely as
$$f(x)=\sum_{k=0}^{\lrf{{n}/{2}}}\gamma_kx^k(1+x)^{n-2k}.$$
We say that $f(x)$ is {\it $\gamma$-positive}
if $\gamma_k\geq 0$ for $0\leq k\leq \lrf{{n}/{2}}$ (see~\cite{Branden08,Gal05,Lin15} for instance).
The $\gamma$-positivity of $f(x)$ implies symmetry and unimodality of $f(x)$.
We refer the reader to Athanasiadis's survey article~\cite{Athanasiadis17} for details.
This paper is concerned with refinements of the $\gamma$-positivity of Eulerian polynomials.
It is often useful to consider multivariate refinements of enumerative polynomials.
In many instances with the help of such refinements more connections among various statistics can be discovered.

Let $[n]=\{1,2,\ldots,n\}$.
Let $\msn$ denote the symmetric group of all permutations of $[n]$ and let $\pi=\pi(1)\pi(2)\cdots\pi(n)\in\msn$.
A {\it descent} (resp.~{\it ascent, excedance}) of $\pi$ is an index $i\in[n-1]$
such that $\pi(i)>\pi(i+1)$ (resp.~$\pi(i)<\pi(i+1)$, $\pi(i)>i$). Let $\des(\pi)$ (resp.~$\asc(\pi)$, $\exc(\pi)$) denote the number of descents (resp.~ascents, excedances) of $\pi$.
It is well known that descents, ascents and excedances are equidistributed over the symmetric group,
and their common enumerative polynomials are the {\it Eulerian polynomial} $A_n(x)$, i.e.,
$$A_n(x)=\sum_{\pi\in\msn}x^{\des(\pi)}=\sum_{\pi\in\msn}x^{\asc(\pi)}=\sum_{\pi\in\msn}x^{\exc(\pi)}.$$
The exponential generating function of the polynomials $A_{n}(x)$ is given as follows:
\begin{equation}\label{Ankx-deff}
A(x;z)=\sum_{n=0}^\infty A_{n}(x)\frac{z^n}{n!}=\frac{x-1}{x-e^{(x-1)z}}.
\end{equation}

An index $i$ is called a {\it double descent} of $\pi$ if $\pi(i-1)>\pi(i)>\pi(i+1)$, where $\pi(0)=\pi(n+1)=0$.
Foata and Sch\"utzenberger~\cite{Foata70} obtained that
\begin{equation}\label{Anx-gamma}
A_n(x)=\sum_{j=0}^{(n-1)/2}\gamma_{n,j}x^i(1+x)^{n-1-2i},
\end{equation}
where $\gamma_{n,j}$ counts permutations $\pi\in\msn$ which have no double descents and $\des(\pi)=j$.
This statement, along several multivariate refinements and generalizations, were frequently discovered in the past decades, see~\cite{Athanasiadis17,Ma19,Zhuang17} and references therein.
For example, by using the theory of enriched $P$-partitions, Stembridge~\cite[Remark 4.8]{Stembridge97} showed that
\begin{equation}\label{AnxWni}
A_n(x)=\frac{1}{2^{n-1}}\sum_{i=0}^{(n-1)/2}4^iW(n,i)x^i(1+x)^{n-1-2i},
\end{equation}
where $W(n,i)$ is the the number of permutations in $\msn$ with $i$ {\it interior peaks}, i.e., the indices $i\in\{2,\ldots,n-1\}$ such that $\pi(i-1)<\pi(i)>\pi(i+1)$.
By using modified Foata-Strehl action~\cite{Branden08}, we see that the expansion~\eqref{AnxWni} is equivalent to~\eqref{Anx-gamma}.

We say that $\pi\in\msn$ has no {\it proper double descents} if there is no
index $i\in [n-2]$ such that $\pi(i)>\pi(i+1)>\pi(i+2)$.
The permutation $\pi$ is called {\it simsun} if for all $k$, the
subword of $\pi$ restricted to $[k]$ (in the order
they appear in $\pi$) contains no proper double descents (see~\cite[p.~267]{Sundaram1994}).
Let $\rs_n$ be the set of simsun permutations of length $n$.
The descent polynomials of simsun permutations are defined by $$S_n(x)=\sum_{\pi\in\rs_n}x^{\des(\pi)}=\sum_{i=0}^{\lrf{n/2}}S(n,i)x^i.$$
It should be noted that the polynomial $S_n(x)$ equals the descent polynomial of And\'re permutations of the second kind of order $n+1$, see~\cite{Chow11,Foata01} for details.

\begin{definition}
A value $x=\pi(i)$ is called a {\it cycle double ascent} of $\pi$ if
$i=\pi^{-1}(x)<x<\pi(x)$. Let $\cda(\pi)$ be the number of {\it cycle double ascents} of $\pi$.
 We say that $\pi\in\msn$ is a {\it simsun permutation of the second kind} if for all $k\in [n]$,
after removing the $k$ largest letters of $\pi$, the resulting permutation has no cycle double ascents.
\end{definition}
For example, $(1,6,5,3,4)(2)$ is not a simsun permutation of the second kind since when we remove the letters 5 and 6,
the resulting permutation $(1,3,4)(2)$ contains the cycle double ascent 3.
Let $\rss_n$ be the set of the simsun permutations of the second kind of length $n$.
It follows from~\cite[Eq.~(18)]{Ma17} that
$$\#\{\pi\in\rs_n: \des(\pi)=i\}=\#\{\pi\in\rss_n: \exc(\pi)=i\}.$$
A constructive proof of the following identity was given in~\cite[Proposition~1]{Ma17}:
\begin{equation}\label{WnkSnk}
W(n+1,i)=2^{n-i}S(n,i).
\end{equation}
Combining~\eqref{AnxWni} and~\eqref{WnkSnk}, we get the following well known result.
\begin{proposition}[{\cite{Athanasiadis17,Foata73,Foata01}}]\label{Foata}
For $n\geq 1$, we have
\begin{equation}\label{An1x}
A_{n+1}(x)=\sum_{i=0}^{\lrf{n/2}}S(n,i)(2x)^i(x+1)^{n-2i}.
\end{equation}
\end{proposition}

A {\it fixed point} of $\pi$ is an index $k\in [n]$ such that $\pi(k)=k$.
Let $\fix(\pi)$ denote the number of fixed points of $\pi$.
A permutation $\pi\in\msn$ is a {\it derangement} if it has no fixed points, i.e., $\pi(i)\neq i$ for all $i\in [n]$.
Let $\mdn_n$ be the set of derangements in $\msn$.
The {\it derangement polynomials} are defined by $$d_n(x)=\sum_{\pi\in\mdn_n}x^{\exc(\pi)}.$$
The generating function of $d_n(x)$ is given as follows (see~\cite[Proposition~6]{Brenti90}):
\begin{equation}\label{dxz-EGF}
d(x,z)=\sum_{n=0}^\infty d_n(x)\frac{z^n}{n!}=\frac{1-x}{e^{xz}-xe^{z}}.
\end{equation}
Let $\cyc(\pi)$ be the number of {\it cycles} of $\pi$, and let $\mdn_{n,k}=\{\pi\in\mdn_n: \cda(\pi)=0,~\exc(\pi)=k\}$.
Shin and Zeng~\cite[Theorem~11]{Zeng12} proved the following result by using continued fractions.
\begin{proposition}\label{Zeng}
For $n\geq 2$, we have
$$\sum_{\pi\in\mdn_n}x^{\exc(\pi)}q^{\cyc(\pi)}=\sum_{k=1}^{\lrf{n/2}}\left(\sum_{\pi\in \mdn_{n,k}}q^{\cyc(\pi)}\right)x^{k}(1+x)^{n-2k}.$$
\end{proposition}

The fixed point is closely related to the succession statistic.
A {\it succession} of $\pi\in\msn$ is an index $k\in [n-1]$ such that $\pi(k+1)=\pi(k)+1$. Let $\suc(\pi)$ denote the number of successions of $\pi$.
Diaconis, Evans and Graham~\cite{Diaconis14} call the pair $(k,k+1)$ an {\it unseparated pair} of a permutation $\pi$ if $k$ is a succession, and
gave three different proof of the following result.
\begin{proposition}
For all $I\subseteq [n-1]$, we have
\begin{align*}
&\#\{\pi\in\msn: \{k\in [n-1]: \pi(k+1)=\pi(k)+1\}= I\}\\
&=\#\{\pi\in\msn: \{k\in [n-1]: \pi(k)=k\}= I\}.
\end{align*}
\end{proposition}
Subsequently, Brenti and Marietti~\cite{Brenti18} studied the fixed point and succession statistics of colored permutations in the complex reflection group.
In this paper, we shall establish some connections between successions and fixed points by using some multivariate polynomials.

Let $$P(n,r,s)=\#\{\pi\in\msn: \asc(\pi)=r,~\suc(\pi)=s\}.$$
Roselle~\cite[Eq.~(2.1)]{Roselle68} proved that
$$P(n,r,s)=\binom{n-1}{s}P(n-s,r-s,0).$$
Let $P^*(n,r)$ be the number of permutations of $\msn$ with $r$ ascents, no successions and $\pi(1)>1$.
Let $P^*_n(x)=\sum_{k=0}^{n-1}P^*(n,r)x^r$.
According to~\cite[Eq.~(4.3)]{Roselle68}, we have
\begin{equation}\label{EGF1}
\sum_{n=0}^{\infty}P^*_n(x)\frac{z^n}{n!}=e^{-xz}\left(1+x\sum_{n=1}^{\infty}A_n(x)\frac{z^n}{n!}\right)=\frac{1-x}{e^{xz}-xe^{z}}.
\end{equation}
Comparing~\eqref{dxz-EGF} with~\eqref{EGF1}, we see that $P^*_n(x)=d_n(x)$.
Let $P_n(x)=\sum_{r=0}^{n-1}P(n,r)x^r$, where $P(n,r)=P(n,r,0)$. Roselle~\cite[Eq.~(3.8)]{Roselle68} showed that
\begin{equation}\label{Pnx-Pnx}
xP_n(x)=P^*_n(x)+xP^*_{n-1}(x).
\end{equation}
Motivated by~\eqref{Pnx-Pnx}, it is natural to refine the formula~\eqref{An1x} by using
the succession statistic.

In the next section, we first count permutations in $\msn$ by the numbers of big ascents, descents and successions.
We then consider the joint distribution of sextuple statistics on $B_n$.
In particular, in the proof Theorem~\ref{thm03}, we find the following result.
\begin{theorem}
For $1\leq i\leq n$, let $\widetilde{\mdn}_{n,i}^B$ be the set of type $B$ derangements of order $n$ with the restriction that
the set of negative entries of these derangements is $\{\overline{n},\overline{n-1},\ldots,\overline{n-i+1}\}$.
Let $d_{n,i}^B(x)$ be the excedance polynomials over $\widetilde{\mdn}_{n,i}^B$. Let $d_{n,0}^B(x)=d_n(x)$.
For any $1\leq i\leq n$, we have
$$d_{n,i}^B(x)=d_{n,i-1}^B(x)+d_{n-1,i-1}^B(x).$$
\end{theorem}
In the other sections, we shall prove some main results given in the next section.
\section{Main results}
\subsection{Eulerian polynomials}
\hspace*{\parindent}

A {\it big ascent} in a permutation $\pi\in\msn$ is an index $i$ such that $\pi(i+1)\geq \pi(i)+2$.
Let $\basc(\pi)$ be the number of big ascents of $\pi$. Clearly, $\asc(\pi)=\basc(\pi)+\suc(\pi)$.
Consider the following polynomials
$$A_n(x,y,s)=\sum_{\pi\in\msn}x^{\basc(\pi)}y^{\des(\pi)}s^{\suc(\pi)}.$$
We define $A_0(x,y,s)=1$. In particular, we have $A_n(x)=A_n(x,1,x)=A_n(1,x,1)$.
Below are the polynomials $A_n(x,y,s)$ for $1\leq n\leq 5$:
\begin{align*}
A_1(x,y,s)&=1,~
A_2(x,y,s)=s+y,~
A_3(x,y,s)=(s+y)^2+2xy,\\
A_4(x,y,s)&=(s+y)^3+6xy(s+y)+2xy(x+y),\\
A_5(x,y,s)&=(s+y)^4+12xy(s+y)^2+8xy(s+y)(x+y)+2xy(x+y)^2+16x^2y^2.
\end{align*}

Let $$A(x,y,s;z)=\sum_{n=0}^{\infty}A_{n+1}(x,y,s)\frac{z^n}{n!}.$$
Now we present the first main result of this paper.
\begin{theorem}\label{mainthm01}
We have
\begin{equation}\label{A-EGF}
A(x,y,s;z)=e^{z(y+s)}\left(\frac{y-x}{ye^{xz}-xe^{yz}}\right)^2.
\end{equation}
Moreover, for $n\geq 0$, we have
\begin{equation}\label{Anxys-gamma}
A_{n+1}(x,y,s)=\sum_{i=0}^n(s+y)^i\sum_{j=0}^{\lrf{(n-i)/2}}2^j\gamma_{n,i,j}(xy)^j(x+y)^{n-i-2j},
\end{equation}
where the numbers $\gamma_{n,i,j}$ satisfy the recurrence relation
\begin{equation}\label{gammanij-recu}
\gamma_{n+1,i,j}=\gamma_{n,i-1,j}+(1+i)\gamma_{n,i+1,j-1}+j\gamma_{n,i,j}+(n-i-2j+2)\gamma_{n,i,j-1},
\end{equation}
with the initial conditions $\gamma_{0,0,0}=1$ and $\gamma_{0,i,j}=0$ for $(i,j)\neq (0,0)$.
A combinatorial interpretation of $\gamma_{n,i,j}$ is given as follows:
\begin{equation}\label{gammanij-comb}
\gamma_{n,i,j}=\#\{\pi\in\rss_n: \fix(\pi)=i,~\exc(\pi)=j\}.
\end{equation}
In other words, the number $\gamma_{n,i,j}$ counts the number of simsun permutations of the second kind of order $n$ which have exactly $i$ fixed points and $j$ excedances.
\end{theorem}

Recall that $P_n(x)=A_n(x,1,0)$.
It follows from~\eqref{A-EGF} that
$$A(x,1,0;z)=\sum_{n=0}^{\infty}P_{n+1}(x)\frac{z^n}{n!}=e^{z}\left(\frac{1-x}{e^{xz}-xe^{z}}\right)^2.$$
By using~\eqref{Ankx-deff}, we see that
$$A(x,1,1;z)=\left(\frac{x-1}{x-e^{(x-1)z}}\right)^2=A^2(x;z).$$

A {\it anti-excedance} of $\pi\in\msn$ is an index $i\in[n-1]$ such that $\pi(i)<i$. Let $\aexc(\pi)$ denote the number of anti-excedances of $\pi$.
Note that $\exc(\pi)+\aexc(\pi)+\fix(\pi)=n$.
We define
$$d_n(x,y)=\sum_{\pi\in\mdn_n}x^{\exc(\pi)}y^{\aexc(\pi)},$$
$$d_n(x,y,s)=\sum_{\pi\in\msn}x^{\exc(\pi)}y^{\aexc(\pi)}s^{\fix(\pi)}.$$
It follows from~\eqref{dxz-EGF} that
\begin{equation*}
d(x,y;z)=\sum_{n=0}^{\infty}d_n(x,y)\frac{z^n}{n!}=\frac{y-x}{ye^{xz}-xe^{yz}}.
\end{equation*}
Note that $d_n(x,y,s)=\sum_{i=0}^n\binom{n}{i}s^id_{n-i}(x,y)$. Hence
\begin{equation}\label{dxys-EGF}
d(x,y,s;z)=\sum_{n=0}^{\infty}d_n(x,y,s)\frac{z^n}{n!}=\frac{(y-x)e^{sz}}{ye^{xz}-xe^{yz}}.
\end{equation}
Define
$$\widetilde{A}_n(x,y)=\sum_{\pi\in\msn}x^{\asc(\pi)}y^{\des(\pi)+1},$$
and $\widetilde{A}_0(x,y)=1$. Using~\eqref{Ankx-deff}, it is routine to verify that
$$\sum_{n=0}^\infty\widetilde{A}_n(x,y)\frac{z^n}{n!}=\frac{(y-x)e^{yz}}{ye^{xz}-xe^{yz}}.$$
Then combining this with~\eqref{A-EGF} and~\eqref{dxys-EGF}, we get the following corollary, which gives a connection between the number of successions in $\ms_{n+1}$ and the number of fixed points in $\msn$.
\begin{corollary}
For $n\geq 0$, we have
\begin{equation*}
A_{n+1}(x,y,s)=\sum_{i=0}^n\binom{n}{i}\widetilde{A}_i(x,y)d_{n-i}(x,y,s),
\end{equation*}
In particular, $A_{n+1}(x,1,0)=\sum_{i=0}^n\binom{n}{i}A_i(x)d_{n-i}(x)$.
\end{corollary}

Setting $y=1$ in~\eqref{Anxys-gamma}, we see that
$$\sum_{\pi\in\ms_{n+1}}x^{\basc(\pi)}s^{\suc(\pi)}=\sum_{i=0}^n(1+s)^i\sum_{j=0}^{\lrf{(n-i)/2}}\gamma_{n,i,j}(2x)^j(1+x)^{n-i-2j},$$
and it reduces to~\eqref{An1x} when $s=x$.
We define
$$M(s,x;z)=\sum_{n=0}^\infty\sum_{i=0}^n\sum_{j=0}^{\lrf{(n-i)/2}}\gamma_{n,i,j}s^ix^j\frac{z^n}{n!}.$$
Then
$$M(s,x;z)=\sum_{n=0}^\infty\sum_{\pi\in\rss_n}s^{\fix(\pi)}x^{\exc(\pi)} \frac{z^n}{n!}.$$
Then from~\cite[p.~13]{Ma17}, we see that
\begin{equation}\label{Msxz-explicit}
M(s,x;z)=e^{z(s-1)}\left(\frac{\sqrt{2x-1}\sec\left(\frac{z}{2}\sqrt{2x-1}\right)}
{\sqrt{2x-1}-\tan\left(\frac{z}{2}\sqrt{2x-1}\right)}\right)^2.
\end{equation}
\subsection{Eulerian polynomials of type $B$}
\hspace*{\parindent}

Let $\pm[n]=[n]\cup\{-1,\ldots,-n\}$.
Let $B_n$ be the {\it hyperoctahedral group} of rank $n$. Let $\sigma=\sigma(1)\sigma(2)\cdots\sigma(n)\in B_n$.
Elements of $B_n$ are signed permutations of $\pm[n]$ with the property that $\sigma(-i)=-\sigma(i)$ for all $i\in [n]$.
Let $\des_B(\sigma)=\#\{i\in\{0,1,\ldots,n-1\}\mid \sigma(i)>\sigma({i+1})\}$, where $\sigma(0)=0$.
We say that $i\in [n]$ is a {\it weak excedance} of $\sigma$ if $\sigma(i)=i$ or $\sigma(|\sigma(i)|)>\sigma(i)$ (see~\cite[p.~431]{Brenti94}).
Let $\we(\sigma)$ be the number of weak excedances of $\sigma$.
According to~\cite[Theorem~3.15]{Brenti94}, the statistics $\des_B$ and $\we$ have the same distribution over $B_n$,
and their common enumerative polynomial is the
{\it Eulerian polynomial of type $B$}, i.e.,
$$B_n(x)=\sum_{\sigma\in B_n}x^{\des_B(\sigma)}=\sum_{\sigma\in B_n}x^{\we(\sigma)}.$$

A {\it left peak} of $\pi\in\msn$ is an index $i\in[n-1]$ such that $\pi(i-1)<\pi(i)>\pi(i+1)$, where we take $\pi(0)=0$.
Let $\lpk(\pi)$ be the number of left peaks in $\pi$.
Let $Q(n,i)=\{\pi\in\msn: \lpk(\pi)=i\}$.
By using the theory of enriched $P$-partitions, Petersen~\cite[Proposition~4.15]{Petersen07} obtained that
\begin{equation}\label{Bnxgamma}
B_n(x)=\sum_{i=0}^{\lrf{n/2}}4^iQ(n,i)x^i(1+x)^{n-2i},
\end{equation}
which has been extensively studied in recent years, see~\cite{Chow08,Lin15,Zeng16,Zhuang17} and references therein.
Let $b_n(x)=\sum_{i=0}^{\lrf{n/2}}4^iQ(n,i)x^i$.
According to~\cite[Proposition~4.10]{Chow08}, we have
\begin{equation}\label{bxz-EGF}
b(x;z)=1+\sum_{n=1}^\infty b_n(x)\frac{z^n}{n!}=\frac{\sqrt{4x-1}\sec\left(z\sqrt{4x-1}\right)}
{\sqrt{4x-1}-\tan\left(z\sqrt{4x-1}\right)}.
\end{equation}

Denote by $\overline{i}$ the negative element $-i$.
We say that $i$ is an {\it excedance} (resp.~{\it anti-excedance}, {\it fixed point}, {\it singleton}) of $\sigma$ if $\sigma(|\sigma(i)|)>\sigma(i)$ (resp.~$\sigma(|\sigma(i)|)<\sigma(i)$, $\sigma(i)=i$, $\sigma(i)=\overline{i}$).
Let $\exc(\sigma)$ (resp.~$\aexc(\sigma)$, $\fix(\sigma)$, $\ston(\sigma)$, $N(\sigma)$) be the number of excedances (resp.~anti-excedances, fixed points, singletons, negative entries) of $\sigma$.
Consider the following polynomials
$$B_n(x,y,s,t,p,q)=\sum_{\sigma\in B_n}x^{\exc(\sigma)}y^{\aexc(\sigma)}s^{\fix(\sigma)}t^{\ston(\sigma)}p^{\cyc(\sigma)}q^{N(\sigma)}.$$
Clearly, $A_n(x)=B_n(x,1,1,1,1,0)$ and $B_n(x)=B_n(x,1,x,1,1,1)$.
Below are the polynomials $B_n(x,y,s,t,p,q)$ for $0\leq n\leq 3$:
\begin{align*}
B_0(x,y,s,t,p,q)&=1,\\
B_1(x,y,s,t,p,q)&=p(s+qt),\\
B_2(x,y,s,t,p,q)&=p^2(s+qt)^2+p(1+q)^2xy,\\
B_3(x,y,s,t,p,q)&=p^3(s+qt)^3+3p^2(1+q)^2(s+qt)xy+p(1+q)^3xy(x+y).
\end{align*}

Now we give the second main result of this paper.
\begin{theorem}\label{mainthm02}
We have
\begin{equation}\label{Bxy-EGF}
B(x,y,s,t,p,q;z)=\sum_{n=0}^\infty B_n(x,y,s,t,p,q)\frac{z^n}{n!}=\left(\frac{(y-x)e^{(s+qt)z}}{ye^{(1+q)xz}-xe^{(1+q)yz}}\right)^p.
\end{equation}
For $n\geq 0$, we have
\begin{equation}\label{Bnxyspq-gamma}
B_n(x,y,s,t,p,q)=\sum_{i=0}^n(s+qt)^i(1+q)^{n-i}\sum_{j=0}^{\lrf{(n-i)/2}}b_{n,i,j}(p)(xy)^j(x+y)^{n-i-2j},
\end{equation}
where the coefficients $b_{n,i,j}(p)$ satisfy the recurrence relation
\begin{equation}\label{bnij-recu}
b_{n+1,i,j}(p)=pb_{n,i-1,j}(p)+(1+i)b_{n,i+1,j-1}(p)+jb_{n,i,j}(p)+2(n-i-2j+2)b_{n,i,j-1}(p),
\end{equation}
with the initial conditions $b_{0,0,0}(p)=1$ and $b_{0,i,j}(p)=0$ for $(i,j)\neq (0,0)$.
Moreover, we have
\begin{equation}\label{bnij-combin}
b_{n,i,j}(p)=\sum_{\pi\in \ms_{n,i,j}}p^{\cyc(\pi)}.
\end{equation}
where $\ms_{n,i,j}=\{\pi\in\msn: \cda(\pi)=0,~\fix(\pi)=i,~\exc(\pi)=j\}$.
In other words, the number $b_{n,i,j}(1)$ counts permutations in $\msn$ with no cycle double ascents, $i$ fixed pints and $j$ excedances.
\end{theorem}

It should be noted that~\eqref{Bnxyspq-gamma} is comparable with the $\gamma$-positivity result of Shin and Zeng~\cite[Eq.~(4.9)]{Zeng16},
and the expansion formula~\eqref{Bnxgamma} and Proposition~\ref{Zeng} are both special cases of~\eqref{Bnxyspq-gamma}.

Note that
$B(x,y,s,y,1,1;z)=\sqrt{A(x,y,s;2z)}$,
where $A(x,y,s;z)$ is given by~\eqref{A-EGF}. So we immediately get the following result, which gives a connection between the number of successions in $\ms_{n+1}$ and the number of fixed points in $B_n$.
\begin{corollary}
For $n\geq 0$, we have
$$2^nA_{n+1}(x,y,s)=\sum_{i=0}^n\binom{n}{i}B_i(x,y,s,y,1,1)B_{n-i}(x,y,s,y,1,1).$$
\end{corollary}

Set $b_{n,i,j}=b_{n,i,j}(1)$.
We define
$$C(s,x;z)=\sum_{n=0}^\infty\sum_{i=0}^n\sum_{j=0}^{\lrf{(n-i)/2}}b_{n,i,j}s^ix^j\frac{z^n}{n!}.$$
Then
$$C(s,x;z)=\sum_{n=0}^\infty\sum_{\pi\in\mc_n}s^{\fix(\pi)}x^{\exc(\pi)} \frac{z^n}{n!},$$
where $\mc_n=\{\pi\in\msn: \cda(\pi)=0\}$.
We now provide an explicit formula for $C(s,x;z)$.
\begin{theorem}\label{Cxzthm}
We have
\begin{equation}\label{Csxz-explicit}
C(s,x;z)=e^{z\left(s-\frac{1}{2}\right)}\frac{\sqrt{4x-1}\sec\left(\frac{z}{2}\sqrt{4x-1}\right)}
{\sqrt{4x-1}-\tan\left(\frac{z}{2}\sqrt{4x-1}\right)}.
\end{equation}
\end{theorem}
Comparing~\eqref{Csxz-explicit} with~\eqref{Msxz-explicit} and~\eqref{bxz-EGF}, we immediately get the following corollary.
\begin{corollary}
We have
\begin{equation}\label{Cxzbxz}
C(1/2,x;2z)=b(x;z),
\end{equation}
$$C^2(s,x;z)=e^{sz}M(s,2x;z).$$
\end{corollary}

It should be noted that an equivalent formula of~\eqref{Cxzbxz} is given as follows:
\begin{equation*}
\sum_{\pi\in\mc_n}2^{n-\fix(\pi)}x^{\exc(\pi)}=\sum_{\pi\in\msn}(4x)^{\lpk(\pi)}.
\end{equation*}

Define  $$b_n(x)=\sum_{j=0}^{\lrf{n/2}}\sum_{i=0}^nb_{n,i,j}x^j(1+x)^{n-2j}.$$
It is routine to check that
$$\sum_{n=0}^\infty b_n(x) \frac{z^n}{n!}=C\left(1,\frac{x}{(1+x)^2};(1+x)z\right)=\frac{(x-1)e^{xz}}{x-e^{(x-1)z}}.$$
Thus $$b_n(x)=\sum_{i=0}^n\binom{n}{i}A_i(x)x^{n-i},$$
which is the $h$-polynomial of the stellahedra~\cite{Postnikov06}.
Moreover, one can easily verify that
$$C\left(0,\frac{x}{(1+x)^2};(1+x)z\right)=\frac{1-x}{e^{xz}-xe^{z}},$$
which is the exponential generating function of the derangement polynomials $d_n(x)$.

In the rest part of this subsection, we shall
consider the following multivariate polynomials
$$B_n(x,y,s,t,q)=\sum_{\sigma\in B_n}x^{\exc(\sigma)}y^{\aexc(\sigma)}s^{\fix(\sigma)}t^{\ston(\sigma)}q^{N(\sigma)},~B_0(x,y,s,t,q)=1.$$
Note that $B_n(x,y,s,t,q)=B_n(x,y,s,t,1,q)$.
In particular, we have
$$A_n(x)=B_n(x,1,1,0,0),~d_n(x)=B_n(x,1,0,0,0),$$
$$B_n(x)=B_n(x,1,x,1,1),~d_n^B(x)=B_n(x,1,0,1,1),$$
where $d_n^B(x)$ are the type $B$ derangement polynomials that will be discussed in the next subsection.
If follows from~\eqref{Bxy-EGF} that $$B(x,y,s,t,q;z)=\sum_{n=0}^\infty B_n(x,y,s,t,q)\frac{z^n}{n!}=\frac{(y-x)e^{(s+qt)z}}{ye^{(1+q)xz}-xe^{(1+q)yz}}.$$
The first few $B_n(x,y,s,t,q)$ are given as follows:
\begin{align*}
B_1(x,y,s,t,q)&=s+qt,\\
B_2(x,y,s,t,q)&=(s+qt)^2+(1+q)^2xy,\\
B_3(x,y,s,t,q)&=(s+qt)^3+3(1+q)^2(s+qt)xy+(1+q)^3xy(x+y).
\end{align*}

Define $\Phi_{0}(x,y)=\Phi_{1}(x,y)=0$, and
$$\Phi_{n}(x,y)=xy\frac{x^{n-1}-y^{n-1}}{x-y}=xy(x^{n-2}+x^{n-3}y+\cdots+xy^{n-3}+y^{n-2}) ~~\text{for $n\geq 2$}.$$
In particular, $\Phi_{2}(x,y)=xy$ and $\Phi_{3}(x,y)=xy(x+y)$.
We now present the following result.
\begin{theorem}\label{Bnxystq-recu}
For $n\geq 2$, we have
\begin{equation}\label{Bnxystq}
B_n(x,y,s,t,q)=(s+qt)^n+\sum_{k=0}^{n-2}\binom{n}{k}B_k(x,y,s,t,q)\Phi_{n-k}(x,y)(1+q)^{n-k},
\end{equation}
\end{theorem}
\begin{proof}
It is easy to verify that
$$\Phi(x,y,q;z):=\sum_{n=0}^\infty\Phi_{n}(x,y)(1+q)^n\frac{z^n}{n!}=1-\frac{ye^{(1+q)xz}-xe^{(1+q)yz}}{y-x}.$$
For $n\geq 2$, we define
\begin{equation}\label{fnxystq}
f_n(x,y,s,t,q)=(s+qt)^n+\sum_{k=0}^{n-2}\binom{n}{k}B_k(x,y,s,t,q)\Phi_{n-k}(x,y)(1+q)^{n-k}.
\end{equation}
Set $f_0(x,y,s,t,q)=1$ and $f_1(x,y,s,t,q)=s+qt$.
It follows from~\eqref{fnxystq} that
\begin{align*}
f(x,y,s,t,q;z)&=\sum_{n=0}^\infty f_n(x,y,s,t,q)\frac{z^n}{n!}\\
=&e^{(s+qt)z}+B(x,y,s,t,q;z)\Phi(x,y,q;z)\\
&=e^{(s+qt)z}+\left(\frac{(y-x)e^{(s+qt)z}}{ye^{(1+q)xz}-xe^{(1+q)yz}}\right)\left(1-\frac{ye^{(1+q)xz}-xe^{(1+q)yz}}{y-x}\right)\\
&=B(x,y,s,t,q;z),
\end{align*}
which leads to the desired result.
This completes the proof.
\end{proof}
The following corollary is immediate from Theorem~\ref{Bnxystq-recu} by considering some special cases.
\begin{corollary}\label{coro}
For $n\geq 2$, we have
\begin{align*}
A_n(x)&=1+\sum_{k=0}^{n-2}\binom{n}{k}A_k(x)(x+x^2+\cdots+x^{n-1-k});\\
d_n(x)&=\sum_{k=0}^{n-2}\binom{n}{k}d_k(x)(x+x^2+\cdots+x^{n-1-k});\\
B_n(x)&=(1+x)^n+\sum_{k=0}^{n-2}\binom{n}{k}B_k(x)(x+x^2+\cdots+x^{n-1-k})2^{n-k};\\
d_n^B(x)&=1+\sum_{k=0}^{n-2}\binom{n}{k}d_k^B(x)(x+x^2+\cdots+x^{n-1-k})2^{n-k};\\
\end{align*}
\end{corollary}
Very recently, by using the theory of geometric combinatorics, Juhnke-Kubitzke, Murai and Sieg~\cite[Corollary~4.2]{Sieg19}
obtained the recurrence relation of $d_n(x)$ that is given in Corollary~\ref{coro}. It would be interesting to provide a geometric interpretation for
Theorem~\ref{Bnxystq-recu}.
\subsection{Derangement polynomials}
\hspace*{\parindent}

We say that $i$ is an {\it excedance of type $B$} if $\sigma(i)=\overline{i}$ or $\sigma(|\sigma(i)|)>\sigma(i)$ (see~\cite[p.~431]{Brenti94}).
Let $\exc_B(\sigma)$ be the number of excedances of type $B$. Then $\exc_B(\sigma)=\exc(\sigma)+\st(\sigma)$.
A {\it derangement of type $B$} is a
signed permutation $\pi\in B_n$ with no fixed points.
Let $\mdn_n^B$ be the set of all derangements in $B_n$.
In the past decades, the following two kinds of derangement polynomials for $\mdn_n^B$ have been extensively studied:
\begin{align*}
d_n^B(x)=\sum_{\pi\in \mdn_n^B}x^{\exc(\pi)},~
\widetilde{d}_n^B(x)=\sum_{\pi\in \mdn_n^B}x^{\exc_B(\pi)}.
\end{align*}
As pointed out by Chow~\cite[Theorem~5.1]{Chow09}, $d_n^B(x)=x^n\widetilde{d}_n^B(1/x)$. Moreover,
Chow~\cite[Theorem~4.7]{Chow09} showed that $d_n^B(x)$ can be expressed as a sum of certain nonnegative unimodal polynomials.
Chen et al.~\cite[Theorem~4.6]{Chen09} showed that the polynomials $\widetilde{d}_n^B(x)$ possess the spiral property.
Athanasiadis and Savvidou~\cite{Athanasiadis13} proved that there are nonnegative integers $\xi_{n,k}^+$ and $\xi_{n,k}^-$ such that
\begin{equation}\label{Dnbx-gamma}
\widetilde{d}_n^B(x)=\sum_{k\geq 0}\xi_{n,k}^+x^k(1+x)^{n-2k}+\sum_{k\geq 0}\xi_{n,k}^-x^k(1+x)^{n+1-2k}.
\end{equation}

Consider the {\it type $B$ $q$-derangement polynomials}
$$d_n^B(x,q)=\sum_{\sigma\in \mdn_n^B}x^{\exc(\sigma)}q^{N(\sigma)}.$$
Note that $d_n^B(x,0)=d_n(x)$.
Below are the polynomials $d_n^B(x,q)$ for $0\leq n\leq 3$:
\begin{align*}
d_0^B(x,q)&=1,~
d_1^B(x,q)=q,~
d_2^B(x,q)=x+2qx+q^2(1+x),\\
d_3^B(x,q)&=x(1+x)+3qx(2+x)+3q^2x(3+x)+q^3(1+4x+x^2),\\
d_4^B(x,q)&=x(1+7x+x^2)+4qx(2+8x+x^2)+6q^2x(4+9x+x^2)+\\&4q^3x(7+10x+x^2)+q^4(1+11x+11x^2+x^3).
\end{align*}
As special cases of~\eqref{Bxy-EGF} and~\eqref{Bnxyspq-gamma}, we get that
\begin{equation}\label{Dnbxq-EGF}
\sum_{n=0}^\infty d_n^B(x,q)\frac{z^n}{n!}=\frac{(1-x)e^{qz}}{e^{(1+q)xz}-xe^{(1+q)z}},
\end{equation}
\begin{equation}\label{Dnbxq}
d_n^B(x,q)=\sum_{i=0}^nq^i(1+q)^{n-i}\sum_{j=0}^{\lrf{(n-i)/2}}b_{n,i,j}x^j(x+1)^{n-i-2j}.
\end{equation}

For nonnegative integers $m$ and $n$, let $[m,n]=\{m,m+1,\ldots, n\}$.
For integers $n,r\geq 1$, an {\it $r$-colored permutation} can be written as $\pi^c$, where $\pi\in \msn$ and
$c=(c_1,c_2,\ldots,c_n)\in [0,r-1]^n$. As usual, $\pi^c$ can be denoted as $\pi_1^{c_1}\pi_2^{c_2}\cdots\pi_n^{c_n}$, where $c_i$ can
be thought of as the color assigned to $\pi_i$.
Denote by $\z_r\wr\msn$ the set of all $r$-colored permutations of order $n$.
Given an element $\pi^c\in \z_r \wr \msn$.
Following Steingr\'imsson~\cite{Steingrimsson},
we say that an entry $\pi_i^{c_i}$ is an {\it excedance} of $\pi^c$ if $\pi_i>i$ or $\pi_i=i$ and $c_i>0$.
Let $\exc(\pi^c)$ be the number of excedances of $\pi^c$.
A {\it fixed point} of $\pi^c\in \z_r \wr \msn$ is an entry $\pi_k^{c_k}$ such that $\pi_k=k$ and $c_k=0$.
An element $\pi^c\in \z_r \wr \msn$ is called a {\it derangement} if it has no fixed points.
Let $\D_{n,r}$ be the set of derangements in $\z_r \wr \msn$.
The {\it$r$-colored derangement polynomial} is defined by
$$d_{n,r}(x)=\sum_{\pi^c\in \D_{n,r}}x^{\exc(\pi^c)}.$$

There has been much work on the polynomials $d_{n,r}(x)$.
Chow and Toufik~\cite[Proposition~4]{Chow10} found that for $n,r\geq 1$,
$$d_{n,r}(x)=\sum_{\pi\in\msn}(r-1)^{\fix(\pi)}r^{n-\fix(\pi)}x^{\exc(\pi)+\fix(\pi)}.$$
Combining results of Shareshian and Wachs~\cite{Shareshian09} and Linusson, Shareshian and Wachs~\cite{Linusson12} on the homology of Rees products of posets,
Athanasiadis~\cite[Theorem 1.3]{Athanasiadis14} obtained that there are nonnegative integers $\xi^+_{n,r,i}$ and $\xi^-_{n,r,i}$
such that
\begin{equation}\label{dnrx-gamma}
d_{n,r}(x)=\sum_{i=0}^{\lrf{n/2}}\xi^+_{n,r,i}x^i(1+x)^{n-2i}+\sum_{i=1}^{\lrf{(n+1)/2}}\xi^-_{n,r,i}x^i(1+x)^{n+1-2i}.
\end{equation}
Shin and Zeng~\cite{Zeng16} proved that the polynomials $d_{n,r}(x)$ and the flag excedance polynomials for $\D_{n,r}$ have several similar expansion formulas.

According to~\cite[Theorem 5]{Chow10}, we have
\begin{equation}\label{dnrx-EGF}
\sum_{n=0}^\infty d_{n,r}(x)\frac{z^n}{n!}=\frac{(1-x)e^{(r-1)xz}}{e^{rxz}-xe^{rz}}.
\end{equation}
Comparing~\eqref{Dnbxq-EGF} with~\eqref{dnrx-EGF}, we find the following result.
\begin{proposition}\label{propdnbx}
When $q$ is a nonnegative integer, we have
$d_n^B(x,q)=x^nd_{n,q+1}\left({1}/{x}\right)$.
\end{proposition}

From Proposition~\ref{propdnbx}, we see that the polynomials $d_n^B(x,q)$ have the similar expansion as~\eqref{dnrx-gamma}, and when $q=1$,
one can derive~\eqref{Dnbx-gamma}. Here we provide such an expansion.
\begin{theorem}\label{theoremdnbx}
For $n\geq 1$, we have
we have
\begin{equation}\label{fnbigamma}
d_n^B(x,q)=f_{n}^{+}(x,q)+f_n^-(x,q),
\end{equation}
where
\begin{equation*}
f_{n}^{+}(x,q)=\sum_{j=0}^{\lrf{n/2}}f_{n,j}^+(q)x^j(1+x)^{n-2j},~f_n^-(x,q)=\sum_{j=0}^{\lrf{(n-1)/2}}f_{n,j}^-(q)x^j(1+x)^{n-1-2j}.
\end{equation*}
Moreover, the polynomials $f_{n}^{+}(x,q)$ and $f_n^-(x,q)$ satisfy the following recurrence system
\begin{align*}
f_{n+1}^+(x,q)&=n(1+q)xf_{n}^+(x,q)+(1+q)x(1-x)\frac{d}{dx}f_{n}^+(x,q)+\\
&(1+q)nxf_{n-1}^+(x,q)+xf_{n}^-(x,q),\\
f_{n+1}^-(x,q)&=\left(q(1+x)+(n-1)(1+q)x\right)f_{n}^-(x,q)+(1+q)x(1-x)\frac{d}{dx}f_{n}^-(x,q)+\\
&(1+q)nxf_{n-1}^-(x,q)+qf_{n}^+(x,q),
\end{align*}
with the initial conditions $f_{0}^+(x,q)=1$ and $f_{0}^-(x,q)=0$.
When $q\geq 0$, the polynomials $f_{n}^{+}(x,q)$ and $f_{n}^{-}(x,q)$ are both $\gamma$-positive.
\end{theorem}

Below are the polynomials $f_{n}^{+}(x,q)$ and $f_{n}^{-}(x,q)$ for $1\leq n\leq 3$:
\begin{align*}
f_{1}^+(x,q)&=0,~f_{1}^-(x,q)=q,\\
f_{2}^+(x,q)&=(1+2q)x,~f_{2}^-(x,q)=q^2(1+x),\\
f_{3}^+(x,q)&=(1+3q+3q^2)(x+x^2),~f_{3}^-(x,q)=q^3+(3q+6q^2+4q^3)x+q^3x^2.
\end{align*}
In particular, $f_{n}^+(x,0)=d_n(x)$ and $f_{n}^-(x,0)=0$.

As the third main result of this paper, we now give an expansion of $d_n^B(x,q)$ in terms of the derangement polynomials.
\begin{theorem}\label{thm03}
For any $n\geq0$, we have
$$d_n^B(x,q)=\sum_{i=0}^n\sum_{j=0}^i\binom{n}{i}\binom{i}{j}d_{n-j}(x)q^i.$$
\end{theorem}
\begin{proof}
Let
$$d_n^B(x,q)=\sum_{i=0}^n\binom{n}{i}d_{n,i}^B(x)q^i.$$
Clearly, $d_{n,0}^B(x)=d_n(x)$ and $d_{n,n}(x)=A_n(x)$.
In the following discussion, we always write $\sigma\in \mdn_n^B$ in the cycle form.
Note that $$\binom{n}{i}d_{n,i}^B(x)=\sum_{\sigma\in \mdn_{n,i}^B}x^{\exc(\sigma)},$$
where $\mdn_{n,i}^B=\{\sigma\in \mdn_n^B: N(\sigma)=i\}$. For $1\leq i\leq n$,
let $\widetilde{\mdn}_{n,i}^B $ be the set of $\sigma\in\mdn_n^B$ with the restriction that
the set of negative entries of $\sigma$ is $\{\overline{n},\overline{n-1},\ldots,\overline{n-i+1}\}$.
Then we have $$d_{n,i}^B(x)=\sum_{\sigma\in \widetilde{\mdn}_{n,i}^B}x^{\exc(\sigma)}.$$
In order to show that $$d_{n,i}^B=\sum_{j=0}^i\binom{i}{j}d_{n-j}(x),$$
it suffices to show that for any $1\leq i\leq n$, we have
\begin{equation}\label{dniBx}
d_{n,i}^B(x)=d_{n,i-1}^B(x)+d_{n-1,i-1}^B(x).
\end{equation}

For any $1\leq i\leq n$, we partition the set $\widetilde{\mdn}_{n,i}^B$ into three subsets:
\begin{align*}
\widetilde{\mdn}_{n,i}^{B,1}&=\{\sigma\in\widetilde{\mdn}_{n,i}^B\mid \text{$\overline{n}$ is a singleton of $\sigma$}\},\\
\widetilde{\mdn}_{n,i}^{B,2}&=\{\sigma\in\widetilde{\mdn}_{n,i}^B\mid \st(\sigma)=0\},\\
\widetilde{\mdn}_{n,i}^{B,3}&=\{\sigma\in\widetilde{\mdn}_{n,i}^B\mid \st(\sigma)>0 ~\text{and $\overline{n}$ is not a singleton of $\sigma$}\}.
\end{align*}
{\it Claim 1.} There is a bijection $\phi_1:\widetilde{\mdn}_{n,i}^{B,1}\mapsto\widetilde{\mdn}_{n-1,i-1}^B$.
For any $\sigma\in\widetilde{\mdn}_{n,i}^{B,1}$, we define $\phi_1(\sigma)$ by deleting the cycle $(\overline{n})$ in $\sigma$.
Clearly, $\phi_1(\sigma)\in \widetilde{\mdn}_{n-1,i-1}^B$. For any $\sigma'\in\widetilde{\mdn}_{n-1,i-1}^{B}$,
the permutation $\phi_1^{-1}(\sigma')$ is obtained from $\sigma'$ by appending $(\overline{n})$ to $\sigma'$ as a new cycle.

{\it Claim 2.} There is an order-preserving bijection $\phi_2:\widetilde{\mdn}_{n,i}^{B,2}\mapsto \widetilde{\mdn}_{n,i-1}^{B,2}$.
For $\sigma\in\widetilde{\mdn}_{n,i}^{B,2}$, we define the map $\phi_2$ by
$$\phi_2(\sigma)(j)=\left\{\begin{array}{lll}
\sigma(j)+1,&\text{ if $\sigma(j)\in \{1,2,\ldots,n-i\}$;}\\
1,&\text{if $\sigma(j)=\overline{n-i+1}$;}\\
\sigma(j),&\text{if $\sigma(j)\in \{\overline{n-i+2},\ldots,\overline{n-1},\overline{n}\}$.}
\end{array}\right.$$
It is clear that $\phi_2(\sigma)\in \widetilde{\mdn}_{n,i-1}^{B,2}$ and $\exc(\sigma)=\exc(\phi_2(\sigma))$.
For $\sigma'\in \widetilde{\mdn}_{n,i-1}^{B,2}$, the converse of $\phi_2$ is given as follows:
$$\phi_2^{-1}(\sigma')(j)=\left\{\begin{array}{lll}
\sigma(j)-1,&\text{ if $\sigma(j)\in \{2,3,\ldots,n-i+1\}$;}\\
\overline{n-i+1},&\text{if $\sigma'(j)=1$;}\\
\sigma(j),&\text{if $\sigma(j)\in \{\overline{n-i+2},\ldots,\overline{n-1},\overline{n}\}$.}
\end{array}\right.$$

{\it Claim 3.} There is an order-preserving bijection $\phi_3:\widetilde{\mdn}_{n,i}^{B,3}\mapsto\widetilde{\mdn}_{n,i-1}^B\setminus \widetilde{\mdn}_{n,i-1}^{B,2}$.
For $\sigma\in \widetilde{\mdn}_{n,i}^{B,3}$,
let $\ST(\sigma)$ be the set of singletons of $\sigma$.
We let the set of singletons of $\phi_3(\sigma)$ be defined by
$$\ST\left(\phi_3(\sigma)\right)=\{\overline{k+1}:\overline{k}\in\ST(\sigma)\}.$$

Define
\begin{align*}
\mathcal{A}(\sigma)&=\{\overline{n},\overline{n-1},\ldots,\overline{n-i+1}\}\cup\{1,2,\ldots,n-i\}\setminus \ST(\sigma),\\
\mathcal{B}(\sigma)&=\{\overline{n},\overline{n-1},\ldots,\overline{n-i+2}\}\cup\{1,2,\ldots,n-i,n-i+1\}\setminus \ST\left(\phi_3(\sigma)\right).
\end{align*}
We write the elements in $\mathcal{A}(\sigma)$ and $\mathcal{B}(\sigma)$ in increasing order.
If $\sigma(j)$ is the $k$th element of $\mathcal{A}(\sigma)$, then let
$\phi_3(\sigma)(j)$ be the $k$th element of $\mathcal{B}(\sigma)$.

For example, let $\sigma=(1,4,3,\overline{9},\overline{8})(2,5)(\overline{6})(\overline{7})\in\widetilde{\mdn}_{9,4}^{B,3}$.
Then $\ST(\sigma)=\{\overline{6},\overline{7}\}$ and $\ST(\phi_3(\sigma))=\{\overline{7},\overline{8}\}$.
Moreover,
$\mathcal{A}(\sigma)=\{\overline{9},\overline{8},1,2,3,4,5\},~
\mathcal{B}(\sigma)=\{\overline{9},1,2,3,4,5,6\}$.
Then the order-preserving between $\mathcal{A}(\sigma)$ and $\mathcal{B}(\sigma)$ can be illustrated by the following array:
$$\left(
    \begin{array}{ccccccc}
      \overline{9} & \overline{8} & 1 & 2 & 3 & 4 & 5 \\
      \overline{9} & 1 & 2 & 3 & 4 & 5 & 6 \\
    \end{array}
  \right).
$$
Therefore, $\phi_3(\sigma)=(2,5,4,\overline{9},1)(3,6)(\overline{7})(\overline{8})$.
It is clear that $\phi_3(\sigma)$ has at least one singleton.
Along the same way lines, one can define the reverse of $\phi_3$. It should be noted that the order-preserving bijection $\phi_3$ does not change the number of excedances.

In conclusion, we have
\begin{align*}
\sum_{\sigma\in \widetilde{\mdn}_{n,i}^B}x^{\exc(\sigma)}&=\sum_{\sigma\in \widetilde{\mdn}_{n,i}^{B,1}}x^{\exc(\sigma)}
+\sum_{\sigma\in \widetilde{\mdn}_{n,i}^{B,2}}x^{\exc(\sigma)}+\sum_{\sigma\in \widetilde{\mdn}_{n,i}^{B,3}}x^{\exc(\sigma)}\\
&=\sum_{\sigma\in \widetilde{\mdn}_{n-1,i-1}^{B}}x^{\exc(\sigma)}+\sum_{\sigma\in \widetilde{\mdn}_{n,i-1}^{B}}x^{\exc(\sigma)},
\end{align*}
and this leads to~\eqref{dniBx}. This completes the proof.
\end{proof}

Let $\NS_n=\{\pi\in\msn: \suc(\pi)=0\}$. From~\eqref{Pnx-Pnx}, we see that
that $$xP_n(x)=\sum_{\pi\in\NS_n}x^{\asc(\pi)+1}=d_n(x)+xd_{n-1}(x).$$
Since $x^nd_n(1/x)=d_n(x)$, we get $$\sum_{\pi\in\NS_n}x^{\des(\pi)}=d_n(x)+d_{n-1}(x).$$
A special case of Theorem~\ref{thm03} says that $d_{n,1}^B=d_n(x)+d_{n-1}(x)$. 

The rest of this paper is organized as follows. In Section~\ref{Section02}--\ref{Section05},
we shall prove Theorem~\ref{mainthm01}, Theorem~\ref{mainthm02}, Theorem~\ref{Cxzthm} and Theorem~\ref{theoremdnbx}, respectively.
\section{Proof of Theorem~\ref{mainthm01}}\label{Section02}
The main tool of the proof is context-free grammar.
Let $V$ be an alphabet whose letters are regarded as independent commutative
indeterminates. Following Chen~\cite{Chen93},
a {\it context-free grammar} $G$ over $V$ is a set of substitution rules replacing a variable in $V$ by
a formal function of variables in $V$.
The formal derivative $D:=D_G$ with respect to $G$ is defined as a linear operator
such that each substitution rule is treated as the common differential rule. For two formal functions $u$ and $v$,
we have $D(u+v)=D(u)+D(v)$ and $D(uv)=D(u)v+uD(v)$.
For a constant $c$, we have $D(c)=0$.
For a Laurent polynomial $w$ of variables in $V$, let $$\gen(w;z)=\sum_{n=0}^\infty D^n(w)\frac{z^n}{n!}.$$
Then we have $$\gen(uv;z)=\gen(u;z)\gen(v;z),$$
$$\frac{\partial}{\partial z}\gen(u;z)=\gen(D(u);z).$$

The following two definitions will be used repeatedly in our discussion.
\begin{definition}[{\cite{Chen17,Fu18}}]
A {\it grammatical labeling} is an assignment of the underlying elements of a combinatorial structure
with variables, which is consistent with the substitution rules of a grammar.
\end{definition}
\begin{definition}[{\cite{Ma19}}]
A {\it change of grammar} is a substitution method in which the original grammar is replaced with functions of other grammar.
\end{definition}

\begin{lemma}\label{lemmaLM}
If $V=\{L,M,s,x,y\}$ and
\begin{equation}\label{LMxys-G}
G=\{L\rightarrow Ly,M\rightarrow Ms,s\rightarrow xy,x\rightarrow xy,y\rightarrow xy\},
\end{equation}
then we have
\begin{equation}\label{derangment-grammarLMA}
D_{G}^n(LM)=LMA_{n+1}(x,y,s).
\end{equation}
\end{lemma}
\begin{proof}
We now introduce a grammatical labeling of $\pi=\pi(1)\pi(2)\cdots\pi(n)\in \msn$ as follows:
\begin{itemize}
\item [\rm ($i$)]Put a superscript label $L$ at the front of $\pi$;
\item [\rm ($ii$)]Put a superscript label $M$ right after the maximum entry $n$;
  \item [\rm ($iii$)]If $i$ is a big ascent, then put a superscript label $x$ right after $\pi(i)$;
 \item [\rm ($iv$)]If $i$ is a descent and $\pi(i)\neq n$, then put a superscript label $y$ right after $\pi(i)$;
  \item [\rm ($v$)]If $\pi(n)\neq n$, then put a superscript label $y$ at the end of $\pi$;
\item [\rm ($vi$)]If $i$ is a succession, then put a superscript label $s$ right after $\pi(i)$.
\end{itemize}
The weight of $\pi$ is defined to be the product of its labels.
Note that the weight of $\pi$ is given by $$w(\pi)=LMx^{\basc(\pi)}y^{\des(\pi)}s^{\suc(\pi)}.$$
When $n=0,1$, we have $\ms_1=\{^L1^M\}$ and $\ms_2=\{^L1^s2^M, ^L2^M1^y\}$. Note that $D_{G}(LM)=LM(s+y)$.
Hence the weight of the element in $\ms_1$ is $LM$ and the sum of weights of the elements in $\ms_2$ is given by $D_{G}(LM)$.
Suppose we get all labeled permutations in $\pi\in\ms_{n-1}$, where $n\geq 2$. Let
$\widehat{{\pi}}$ be obtained from $\pi\in\ms_{n-1}$ by inserting the entry $n$.
There are six cases to label $n$ and relabel some elements of $\pi$.
The changes of labeling are illustrated as follows:
$$ ^L\pi(1)\cdots (n-1)^M\cdots   \mapsto  ^Ln^M\pi(1)\cdots (n-1)^y\cdots ;$$
$$ ^L\pi(1)\cdots (n-1)^M\cdots   \mapsto ^L\pi(1)\cdots (n-1)^sn^M\cdots;$$
$$ \cdots\pi(i)^x\cdots (n-1)^M\cdots   \mapsto \cdots\pi(i)^xn^M\cdots (n-1)^y\cdots ;$$
$$ \cdots\pi(i)^y\pi(i+1)\cdots(n-1)^M\cdots   \mapsto \cdots\pi(i)^xn^M\pi(i+1)\cdots(n-1)^y\cdots  ;$$
$$ \cdots (n-1)^M\cdots\pi(n-1)^y   \mapsto  \cdots (n-1)^y\cdots\pi(n-1)^xn^M  ;$$
$$ \cdots \pi(i)^s\pi(i+1)\cdots(n-1)^M\cdots   \mapsto \cdots \pi(i)^xn^M\pi(i+1)\cdots (n-1)^y\cdots.$$
In each case, the insertion of $n$ corresponds to one substitution rule in $G$. By induction,
it is routine to check that the action of $D_G$ on elements of $\ms_{n-1}$ generates all elements of $\ms_{n}$.
\end{proof}

\noindent{\bf A proof of
Theorem~\ref{mainthm01}:}
\begin{proof}
(A) Let $G$ be the grammar given by~\eqref{LMxys-G}.
By using~\ref{derangment-grammarLMA}, we see that there exists nonnegative integers $a_{n,i,j}$ such that
$$D_{G}^n(LM)=LM\sum_{i,j=0}^na_{n,i,j}x^iy^js^{n-i-j}.$$
Note that
\begin{align*}
&D_{G}\left(D_{G}^n(LM)\right)\\
&=LM\sum_{i,j=0}^na_{n,i,j}\left(x^iy^{j+1}s^{n-i-j}+x^iy^js^{n+1-i-j}\right)+\\
&LM\sum_{i,j=0}^na_{n,i,j}\left(ix^iy^{j+1}s^{n-i-j}+jx^{i+1}y^{j}s^{n-i-j}+(n-i-j)x^{i+1}y^{j+1}s^{n-1-i-j}\right).
\end{align*}
Comparing the coefficients of $LMx^iy^js^{n+1-i-j}$ in both sides of the above expansion, we get
that 
\begin{equation}\label{anij-recu}
a_{n+1,i,j}=a_{n,i,j}+(1+i)a_{n,i,j-1}+ja_{n,i-1,j}+(n-i-j+2)a_{n,i-1,j-1}.
\end{equation}
Multiplying both sides of~\eqref{anij-recu} by $x^iy^js^{n+1-i-j}$ and summing over all $i,j$, we obtain
\begin{equation}\label{Anxys-recu}
A_{n+2}(x,y,s)=(s+y)A_{n+1}(x,y,s)+xy\left(\frac{\partial}{\partial x}+\frac{\partial}{\partial y}+\frac{\partial}{\partial s}\right)A_{n+1}(x,y,s).
\end{equation}
By rewriting~\eqref{Anxys-recu} in terms of generating function $A:=A(x,y,s;z)$, we have
\begin{equation}\label{recu-Anxys02}
\frac{\partial}{\partial z}A=(s+y)A+xy\left(\frac{\partial}{\partial x}+\frac{\partial}{\partial y}+\frac{\partial}{\partial s}\right)A.
\end{equation}
It is routine to check that the generating function
$$\widetilde{A}=e^{z(y+s)}\left(\frac{y-x}{ye^{xz}-xe^{yz}}\right)^2$$
satisfies~\eqref{recu-Anxys02}. Also, this generating function gives $\widetilde{A}(0,0,0;z)=1,\widetilde{A}(x,0,s;z)=e^{sz}$ and $\widetilde{A}(0,y,s;z)=e^{z(y+s)}$. Hence $A=\widetilde{A}$.

\quad (B)
Setting
$u=xy,v=x+y,t=s+y$ and $I=LM$,
we get
$D_G(u)=uv,D_G(v)=2u,D_G(t)=2u$ and $D_G(I)=It$.
Then a change of the grammar $G$ is given as follows:
\begin{equation}\label{JAcobi-gram02}
G_1=\{I\rightarrow It,t\rightarrow 2u, u\rightarrow uv,v\rightarrow 2u\}.
\end{equation}
Note that $D_{G_1}(I)=It,~D_{G_1}^2(I)=I(t^2+2u)$ and $D_{G_1}^3(I)=I(t^3+6tu+2uv)$.
Then by induction, it is easy to verify that there exist nonnegative integers $\gamma_{n,i,j}$ such that
\begin{equation}\label{DG101}
D_{G_1}^n(I)=I\sum_{i=0}^nt^i\sum_{j=0}^{\lrf{(n-i)/2}}2^j\gamma_{n,i,j}u^jv^{n-i-2j}.
\end{equation}
Then upon taking $u=xy,v=x+y,t=s+y$ and $I=LM$, we get~\eqref{Anxys-gamma}.
In particular, $\gamma_{0,0,0}=1$.
Since $D_{G_1}^{n+1}(I)=D_{G_1}\left(D_{G_1}^n(I)\right)$,
we get
\begin{align*}
&D_{G_1}\left(D_{G_1}^n(I)\right)\\
&=I\sum_{i,j}2^j\gamma_{n,i,j}\left(t^{i+1}u^jv^{n-i-2j}+2it^{i-1}u^{j+1}v^{n-i-2j}\right)+\\
&I\sum_{i,j}2^j\gamma_{n,i,j}\left(jt^iu^jv^{n+1-i-2j}+2(n-i-2j)t^iu^{j+1}v^{n-1-i-2j}\right).
\end{align*}
Comparing the coefficients of $2^jt^iu^jv^{n+1-i-2j}$ in both sides of the above expansion, we get~\eqref{gammanij-recu}.

\quad (C) Now we prove~\eqref{gammanij-comb}. We write any permutation in $\rss_n$ by using its cycle form.
In order to get a permutation $\pi'\in\rss_{n+1}$ with $i$ fixed points and $j$ excedances from a permutation $\pi\in\rss_{n}$,
we distinguish four cases:
\begin{enumerate}
\item [($c_1$)] If $\pi\in \rss_n$ and $\fix(\pi)=i-1$ and $\exc(\pi)=j$, then we need append $(n+1)$ to $\pi$ as a new cycle. This accounts for $\gamma_{n,i-1,j}$ possibilities;
\item [($c_2$)] If $\pi\in \rss_n$ and $\fix(\pi)=i+1$ and $\exc(\pi)=j-1$, then we should insert the entry $n+1$ right after a fixed point. This accounts for $(1+i)\gamma_{n,i+1,j-1}$ possibilities;
 \item [($c_3$)] If $\pi\in \rss_n$ and $\fix(\pi)=i$ and $\exc(\pi)=j$, then we should insert the entry $n+1$ right after an excedance.
This accounts for $j\gamma_{n,i,j}$ possibilities;
 \item [($c_4$)] Since $\pi\in\rss_n$ has no cycle double ascents, we say that $\pi(i)$ is a {\it cycle peak} if $i$ is an excedance, i.e. $i<\pi(i)$. If $\pi\in \rss_n$ and $\fix(\pi)=i$ and $\exc(\pi)=j-1$, then there are $n-i-2(j-1)$ positions could be inserted the entry $n+1$, since
we cannot insert $n+1$ immediately before or right after each cycle peak of $\pi$. Moreover, we cannot insert $n+1$ right after a fixed point.
This accounts for $(n-i-2j+2)\gamma_{n,i,j-1}$ possibilities.
 \end{enumerate}
Thus the claim~\eqref{gammanij-comb} holds. This completes the proof.
\end{proof}
\section{Proof of Theorem~\ref{mainthm02}}\label{Section03}
\hspace*{\parindent}
In the following discussion, we always write permutation, signed or not, by its standard cycle form,
in which each cycle has its smallest (in absolute value) element first and the cycles are written in increasing order of the absolute value of their first elements.
To prove Theorem~\ref{mainthm02}, we need the following lemma.
\begin{lemma}\label{lemmaBn}
If $V=\{J,s,t,x,y\}$ and
\begin{equation}\label{Ixyst-G}
G_2=\{J\rightarrow pJ(s+qt),s\rightarrow (1+q)xy,t\rightarrow (1+q)xy,x\rightarrow (1+q)xy,y\rightarrow (1+q)xy\},
\end{equation}
then we have
\begin{equation}\label{derangment-grammar022}
D_{G_2}^n(J)=JB_n(x,y,s,t,p,q).
\end{equation}
\end{lemma}
\begin{proof}
We first introduce a grammatical labeling of $\sigma\in B_n$ as follows:
\begin{itemize}
  \item [\rm ($L_1$)]If $i$ is an excedance, then put a superscript label $x$ right after $\sigma(i)$;
 \item [\rm ($L_2$)]If $i$ is a anti-excedance, then put a superscript label $y$ right after $\sigma(i)$;
\item [\rm ($L_3$)]If $i$ is a fixed point, then put a superscript label $s$ right after $i$;
\item [\rm ($L_4$)]If $i$ is a singleton, then put a superscript label $t$ right after $i$;
\item [\rm ($L_5$)]Put a superscript label $J$ at the end of $\sigma$;
\item [\rm ($L_6$)]Put a subscript label $p$ at the end of each cycle of $\sigma$;
\item [\rm ($L_7$)]Put a subscript label $q$ right after each negative entry of $\sigma$.
\end{itemize}
For example, let $\sigma=(1,3,\overline{2},6)(\overline{4})(5)$.
The grammatical labeling of $\sigma$ is given below:
$$(1^x3^y\overline{2}^x_q6^y)_p(\overline{4}^t_q)_p(5^s)_p^J.$$
Note that the weight of $\sigma$ is given by $w(\sigma)=Jx^{\exc(\sigma)}y^{\aexc(\sigma)}s^{\fix(\sigma)}t^{\ston(\sigma)}p^{\cyc(\sigma)}q^{N(\sigma)}$.
We proceed by induction on $n$.
For $n=1$, we have $B_1=\{(1^s)^J_p,(\overline{1}^t_q)^J_p\}$.
Note that $D_{G_2}(J)=pJ(s+qt)$.
Hence the result holds for $n=1$. Assume that the result holds for $n$.
Suppose we get all labeled permutations in $\sigma\in B_{n-1}$, where $n\geq 2$. Let
$\widehat{{\sigma}}$ be obtained from $\sigma\in B_{n-1}$ by inserting the entry $n$ or $\overline{n}$.
There are five cases to label the inserted element and relabel some elements of $\sigma$:
\begin{itemize}
  \item [\rm ($c_1$)]If $n$ or $\overline{n}$ appear a new cycle, then the changes of labeling are illustrated as follows:
$$\cdots(\cdots)_p^J\rightarrow\cdots(\cdots)_p(n^s)_p^J,\quad \cdots(\cdots)^J\rightarrow\cdots(\cdots)_p(\overline{n}_q^t)_p^J;$$
 \item [\rm ($c_2$)]If we insert $n$ or $\overline{n}$ right after a fixed point, then the changes of labeling are illustrated as follows:
$$\cdots(i^s)_p(\cdots)\cdots\rightarrow\cdots(i^xn^y)_p(\cdots)\cdots,\quad \cdots(i^s)_p(\cdots)\cdots\rightarrow\cdots(i^y\overline{n}_q^x)_p(\cdots)\cdots;$$
\item [\rm ($c_3$)]If we insert $n$ or $\overline{n}$ right after a singleton, then the changes of labeling are illustrated as follows:
$$\cdots(\overline{i}^t_q)_p(\cdots)\cdots\rightarrow\cdots(\overline{i}^x_qn^y)_p(\cdots)\cdots,\quad \cdots(\overline{i}^t_q)_p(\cdots)\cdots\rightarrow\cdots(\overline{i}^y_q\overline{n}_q^x)_p(\cdots)\cdots;$$
\item [\rm ($c_4$)]If we insert $n$ or $\overline{n}$ right after an excedance, then the changes of labeling are illustrated as follows:
$$\cdots(\cdots \sigma(i)^x\sigma(|\sigma(i)|)\cdots)_p(\cdots)\cdots\rightarrow\cdots(\cdots \sigma(i)^xn^y\sigma(|\sigma(i)|)\cdots)_p(\cdots)\cdots,$$
$$\cdots(\cdots \sigma(i)^x\sigma(|\sigma(i)|)\cdots)_p(\cdots)\cdots\rightarrow\cdots(\cdots \sigma(i)^y\overline{n}_q^x\sigma(|\sigma(i)|)\cdots)_p(\cdots)\cdots;$$
\item [\rm ($c_5$)]If we insert $n$ or $\overline{n}$ right after an anti-excedance, then the changes of labeling are illustrated as follows:
$$\cdots(\cdots \sigma(i)^y\sigma(|\sigma(i)|)\cdots)_p(\cdots)\cdots\rightarrow\cdots(\cdots \sigma(i)^xn^y\sigma(|\sigma(i)|)\cdots)_p(\cdots)\cdots,$$
$$\cdots(\cdots \sigma(i)^y\sigma(|\sigma(i)|)\cdots)_p(\cdots)\cdots\rightarrow\cdots(\cdots \sigma(i)^y\overline{n}_q^x\sigma(|\sigma(i)|)\cdots)_p(\cdots)\cdots.$$
\end{itemize}
In each case, the insertion of $n$ or $\overline{n}$ corresponds to one substitution rule in $G$. By induction,
it is routine to check that the action of $D_{G_2}$ on elements of $B_{n-1}$ generates all elements of $B_{n}$.
\end{proof}

\noindent{\bf A proof of~\eqref{Bxy-EGF}:}
\begin{proof}
Dumont~{\cite[Section~2.1]{Dumont96}} found that
if $V=\{x,y\}$ and
$G=\{x\rightarrow xy, y\rightarrow xy\}$,
then
\begin{equation}\label{Dumont}
D_G^n(x)=\sum_{\pi\in\msn}x^{\exc(\pi)+1}y^{n-\exc(\pi)}
\end{equation} for $n\ge 1$.
From~\eqref{Ixyst-G}, we see that $D_{G_2}(s+qt)=(1+q)^2xy$ and $D_{G_2}(xy)=(1+q)xy(x+y)$. By using~\eqref{Dumont}, it is routine to verify that
\begin{equation*}
D_{G_2}^n(s+qt)=(1+q)^{n+1}\sum_{\pi\in\msn}x^{\exc(\pi)+1}y^{n-\exc(\pi)}
\end{equation*}
for $n\geq 1$,
Then combining this with~\eqref{Ankx-deff}, we find that
\begin{equation*}
\gen(s+qt;z)=\sum_{n=0}^\infty D_{G_2}^n(s+qt)\frac{z^n}{n!}=s+qt+(1+q)xy\frac{e^{(1+q)yz}-e^{(1+q)xz}}{ye^{(1+q)xz}-xe^{(1+q)yz}}.
\end{equation*}
By the {\it Leibniz's rule}, we get
$$D_{G_2}^{n+1}(J)=pD_{G_2}^{n}\left(J(s+qt)\right)=p\sum_{i=0}^n\binom{n}{i}D_{G_2}^{i}(J)D_{G_2}^{n-i}(s+qt).$$
Equivalently, we have
\begin{equation}\label{GenBnxtz}
\frac{\partial}{\partial z}\gen(J;z)=p\gen(J;z)\gen(s+qt;z).
\end{equation}
It is routine to verify that
\begin{align*}
\gen(J;z)&=J\left(\frac{(y-x)e^{z(s+qt)}}{ye^{(1+q)xz}-xe^{(1+q)yz}}\right)^p,
\end{align*}
since this explicit formula satisfies~\eqref{GenBnxtz} and $\gen(J;0)=J$. This completes the proof.
\end{proof}

\noindent{\bf A proof of~\eqref{Bnxyspq-gamma}:}
\begin{proof}
Setting
$u=xy,v=x+y,h=s+qt$,
we get
$D_{G_2}(J)=pJh,D_{G_2}(h)=(1+q)^2u,D_{G_2}(u)=(1+q)uv$ and $D_{G_2}(v)=2(1+q)u$.
Then a change of the grammar $G_2$ is given as follows:
\begin{equation}\label{Change-Bn}
G_3=\{J\rightarrow pJh,h\rightarrow (1+q)^2u, u\rightarrow (1+q)uv,v\rightarrow 2(1+q)u\}.
\end{equation}
Note that $D_{G_3}(J)=pJh,~D_{G_3}^2(J)=J\left(p^2h^2+p(1+q)^2u\right)$ and $$D_{G_3}^3(J)=J\left(p^3h^3+3p^2(1+q)^2hu+p(1+q)^3uv\right).$$
Then by induction, it is routine to verify that there exist polynomials $b_{n,i,j}(q)$ such that
\begin{equation}\label{DG3I}
D_{G_3}^n(J)=J\sum_{i=0}^n(1+q)^{n-i}h^i\sum_{j=0}^{\lrf{(n-i)/2}}b_{n,i,j}(p)u^jv^{n-i-2j}.
\end{equation}
Therefore, applying the operator $D_{G_3}$ to both sides of the above expansion yields
\begin{align*}
D_{G_3}^{n+1}(J)&=D_{G_3}\left(J\sum_{i,j}b_{n,i,j}(p)(1+q)^{n-i}h^iu^jv^{n-i-2j}\right)\\
&=J\sum_{i,j}b_{n,i,j}(p)(1+q)^{n-i}\left(ph^{i+1}u^jv^{n-i-2j}+(1+q)^2ih^{i-1}u^{j+1}v^{n-i-2j}\right)+\\
&J\sum_{i,j}b_{n,i,j}(p)(1+q)^{n+1-i}\left(jh^iu^jv^{n+1-i-2j}+2(n-i-2j)h^iu^{j+1}v^{n-1-i-2j}\right).
\end{align*}
Comparing the coefficients of $I(1+q)^{n+1-i}h^iu^jv^{n+1-i-2j}$ on both sides of the above expansion leads to~\eqref{bnij-recu}.
Then in~\eqref{DG3I}, upon taking $u=xy,v=x+y,h=s+qt$, we get~\eqref{Bnxyspq-gamma}.
In particular, $b_{0,0,0}(p)=1,~b_{1,1,0}(p)=p$ and $b_{1,i,j}(p)=0$ if $(i,j)\neq (1,0)$.
This completes the proof.
\end{proof}

Let $(c_1,c_2,\ldots,c_i)$ be a cycle of $\pi$. Then $c_1=\min\{c_1,\ldots,c_i\}$. Set $c_{i+1}=c_1$. Then $c_j$ is called
\begin{itemize}
  \item a {\it cycle double ascent} in the cycle if $c_{j-1}<c_j<c_{j+1}$, where $2\leq j\leq i-1$;
  \item a {\it cycle double descent} in the cycle if $c_{j-1}>c_j>c_{j+1}$, where $2<j\leq i$;
 \item a {\it cycle peak} in the cycle if $c_{j-1}<c_j>c_{j+1}$, where $2\leq j \leq i$;
 \item a {\it cycle valley} in the cycle if $c_{j-1}>c_j<c_{j+1}$, where $2<j\leq i-1$.
\end{itemize}
We define an action $\varphi_{x}$ on $\msn$ as follows.
Let $c=(c_1,c_2,\ldots,c_i)$ be a cycle of $\pi\in \msn$ with at least two elements.
Consider the following three cases:
\begin{itemize}\item  If $c_k$ is a cycle double ascent in $c$,
then $\varphi_{c_k}(\pi)$ is obtained by deleting $c_k$ and then inserting $c_k$ between $c_j$ and $c_{j+1}$, where $j$ is the smallest index satisfying $k< j\leq i$ and $c_j>c_k>c_{j+1}$;
\item If $c_k$ is a cycle double descent in $c$, then $\varphi_{c_k}(\pi)$ is obtained by deleting $c_k$ and then inserting $c_k$ between $c_j$ and $c_{j+1}$, where $j$ is the largest index satisfying $1\leq j<k$ and $c_j<c_k<c_{j+1}$;
\item If $c_k$ is neither a cycle double ascent nor a cycle double descent in $c$, then $c_k$ is a cycle peak or a cycle valley.
In this case, we let $\varphi_{c_k}(\pi)=\pi$.
\end{itemize}

Following~\cite{Branden08}, we now define a {\it modified Foata-Strehl group action} $\varphi'_x$ on $\msn$ by
$$\varphi'_x(\pi)=\left\{\begin{array}{lll}
\varphi_x(\pi),&\text{ if $x$ is a cycle double ascent or a cycle double descent;}\\
\pi,&\text{if $x$ is a cycle peak or a cycle valley.}\\
\end{array}\right.$$

Define
$$\CDD(\pi)=\{x \mid x\text{~is a cycle double descent of $\pi$}\},$$
$$\ms_{n,i,j,k}^1=\{\pi\in\msn: \cda(\pi)=0,~\fix(\pi)=i,~\exc(\pi)=j,~\cyc(\pi)=k\},$$
$${\ms}_{n,i,j,k}^2=\{\pi\in\msn: \cda(\pi)=1,~\fix(\pi)=i,~\exc(\pi)=j,~\cyc(\pi)=k\}.$$
For $\pi\in\ms_{n,i,j,k}^1$ and $x\in \CDD(\pi)$, it should be noted that $\exc(\pi)$ equals the number of cycle peaks of $\pi$, $\varphi_{x}'(\pi)\in{\ms}_{n,i,j+1,k}^2$ and $x$ is the unique cycle double ascent of $\varphi_{x}'(\pi)$.
Conversely, for $\pi\in{\ms}_{n,i,j+1,k}^2$, let $x$ be the unique cycle double ascent of $\pi$.
Note that $\varphi_{x}'(\pi)\in \ms_{n,i,j,k}^1$ and $x$ becomes a cycle double descent in $\varphi'_{x}(\pi)$. This implies that $$|{\ms}_{n,i,j+1,k}^2|=(n-i-2j)|\ms_{n,i,j,k}^1|.$$

\begin{example}
Let $\pi=(1,10,6,5,7,3,2,8)(4,9)\in\mathfrak{S}_{10,0,4,2}^1$. We have $\CDD(\pi)=\{3,6\}$. Then
\begin{align*}
\varphi_{3}'(\pi)&=(1,3,10,6,5,7,2,8)(4,9),~
\varphi_{6}'(\pi)=(1,6,10,5,7,3,2,8)(4,9),
\end{align*}
and $\varphi_{3}'(\pi),~\varphi_{6}'(\pi)\in\mathfrak{S}_{10,0,5,2}^2$.
\end{example}

\noindent{\bf A proof of~\eqref{bnij-combin}:}
\begin{proof}
In order to get a permutation counted by $b_{n+1,i,j}(p)$, we distinguish five cases:
\begin{enumerate}
\item [($c_1$)] If $\pi\in \ms_{n,i-1,j}$, then we need append $(n+1)$ to $\pi$ as a new cycle. This accounts for the term $pb_{n,i-1,j}(p)$;
\item [($c_2$)] If $\pi\in \ms_{n,i+1,j-1}$, then we should insert the entry $n+1$ right after a fixed point. This accounts for the term $(1+i)b_{n,i+1,j-1}(p)$;
 \item [($c_3$)] If $\pi\in \ms_{n,i,j}$, then we should insert the entry $n+1$ right after an excedance.
This accounts for the term $jb_{n,i,j}(p)$;
 \item [($c_4$)] If $\pi\in \ms_{n,i,j-1,k}^1$, then there are $n-i-2(j-1)$ positions could be inserted the entry $n+1$, since
we cannot insert the entry $n+1$ immediately before or right after each cycle peak. Moreover, we cannot insert the entry $n+1$ right after a fixed point.
This accounts for the term $(n-i-2j+2)b_{n,i,j-1}(p)$;
 \item [($c_5$)]  If $\pi\in \ms_{n,i,j-1,k}^1$, let $x$ be one cycle double descent of $\pi$. Note that
$\varphi'_{x}(\pi)\in \ms_{n,i,j,k}^2$ and $x$ become the unique cycle double ascent. We should insert the entry $n+1$ into $\varphi'_{x}(\pi)$ immediately before $x$.
This accounts for the term $(n-i-2j+2)b_{n,i,j-1}(p)$.
 \end{enumerate}
Thus the claim~\eqref{bnij-combin} holds. This completes the proof.
\end{proof}

\section{Proof of Theorem~\ref{Cxzthm}}\label{Section04}
\hspace*{\parindent}
Consider the following grammar
\begin{equation}\label{GrammaG4}
G_4=\{a\rightarrow qat,t\rightarrow 2u, u\rightarrow uv,v\rightarrow 2u\}.
\end{equation}
Then when $q=1$ and $a=I$, then the grammar $G_4$ reduces to $G_1$, which is defined by~\eqref{JAcobi-gram02}.
\begin{lemma}\label{LemmaGrammaG4}
For the grammar $G_4$ given by~\eqref{GrammaG4}, we have
$$\gen(a;z)=aM^q\left(\frac{t}{v},\frac{2u}{v^2};vz\right),$$
where $M(s,x;z)$ is given by~\eqref{Msxz-explicit}.
\end{lemma}
\begin{proof}
It follows from~\eqref{JAcobi-gram02} and~\eqref{DG101} that
\begin{equation}\label{DG10123}
\frac{\partial}{\partial z}\gen(I;z)=\gen(I;z)\gen(t;z).
\end{equation}
\begin{equation*}\label{DG1012}
\gen(I;z)=IM\left(\frac{t}{v},\frac{2u}{v^2};vz\right),
\end{equation*}
By the {\it Leibniz's rule}, we find that
$$D_{G_4}^{n+1}(a)=qD_{G_4}^n(at)=q\sum_{i=0}^n\binom{n}{i}D_{G_4}^i(a)D_{G_4}^{n-i}(t).$$
Multiplying both sides by $z^n/n!$ and summing over all $n\geq 0$, we get
$$\frac{\partial}{\partial z}\gen(a;z)=q\gen(a;z)\gen(t;z).$$
Combining this with~\eqref{DG10123}, we obtain
\begin{equation}\label{DG14}
\frac{\frac{\partial}{\partial z}\gen(a;z)}{\gen(a;z)}=q\frac{\frac{\partial}{\partial z}\gen(I;z)}{\gen(I;z)}.
\end{equation}
Integrating both sides with respect to $z$ leads to the desired result.
\end{proof}

\noindent{\bf A proof of Theorem~\ref{Cxzthm}:}
\begin{proof}
When $p=1$ and $q=0$, the grammar $G_3$ given by~\eqref{Change-Bn} reduces to the following grammar
\begin{equation*}\label{BnGrammar}
G_5=\{J\rightarrow Jh,h\rightarrow u, u\rightarrow uv,v\rightarrow 2u\}.
\end{equation*}
It follows from~\eqref{DG3I} that
\begin{equation*}\label{DG4I}
D_{G_5}^n(J)=J\sum_{i=0}^nh^i\sum_{j=0}^{\lrf{(n-i)/2}}b_{n,i,j}u^jv^{n-i-2j}.
\end{equation*}
Upon taking $a=J,~q=\frac{1}{2}$ and $h=\frac{t}{2}$, the grammar $G_4$ reduces to $G_5$.
By using Lemma~\ref{LemmaGrammaG4}, we obtain that
\begin{equation}\label{DG4I02}
\gen(J;z)=J\sqrt{M\left(\frac{2h}{v},\frac{2u}{v^2};vz\right)}.
\end{equation}
Setting $v=1,h=s$ and $u=x$ in~\eqref{DG4I02}, we get
$$C(s,x;z)=\sqrt{M\left(2s,2x;z\right)},$$ as desired. This completes the proof.
\end{proof}

\section{Proof of Theorem~\ref{theoremdnbx}}\label{Section05}
We say that $i$ is an {\it weak anti-excedance} of $\sigma$ if $\sigma(|\sigma(i)|)<\sigma(i)$ or $\sigma(i)=\overline{i}$.
Let $\waexc(\sigma)$ be the number of weak anti-excedances of $\sigma$. Then $\waexc(\sigma)=\aexc(\sigma)+\st(\sigma)$.
Consider the following polynomials
$$B_n(x,y,s,q)=\sum_{\sigma\in B_n}x^{\exc(\sigma)}y^{\waexc(\sigma)}s^{\fix(\sigma)}q^{N(\sigma)}.$$
Then $B_n(x,y,s,q)=B_n(x,y,s,y,1,q)$.
When $t=y$ and $p=1$, the grammar~\ref{Ixyst-G} reduces to the following gramar
\begin{equation}\label{G6}
G_6=\{J\rightarrow J(s+qy),s\rightarrow (1+q)xy,x\rightarrow (1+q)xy,y\rightarrow (1+q)xy\}.
\end{equation}
By Lemma~\ref{derangment-grammar022}, we see that
$D_{G_6}^n=JB_n(x,y,s,q)$.
Setting $y=1$ and $s=0$, we obtain
\begin{equation}\label{Dg6yx}
D_{G_6}^n\mid_{y=1,s=0}=Jd_n^B(x,q).
\end{equation}

\noindent{\bf A proof of Theorem~\ref{theoremdnbx}:}
\begin{proof}
Consider a change of the grammar~\eqref{G6}. Setting $H=Jy,~u=xy,~v=x+y$, we see that
$$D_{G_6}(J)=Js+qH,~D_{G_6}(H)=Hs+qHv+Ju,~D_{G_6}(u)=(1+q)uv,~D_{G_6}(v)=2(1+q)u.$$
We now consider the following grammar
\begin{equation*}\label{G7}
G_7=\{J\rightarrow Js+qH,s\rightarrow (1+q)u,~H\rightarrow Hs+qHv+Ju,u\rightarrow (1+q)uv,v\rightarrow 2(1+q)u\}.
\end{equation*}
Note that
$$D_{G_7}^0(J)=J,
~D_{G_7}(J)=Js+qH,~
D_{G_7}^2(J)=J(s^2+(1+2q)u)+H(2qs+q^2v).$$
Then by induction, it is easy to verify that there are polynomials $f_{n,i,j}^+(q)$ and $f_{n,i,j}^-(q)$ such that
\begin{equation}\label{DG7uv}
D_{G_7}^n(J)=J\sum_{i=0}^ns^i\sum_{j=0}^{\lrf{(n-i)/2}}f_{n,i,j}^+(q)u^jv^{n-i-2j}+H\sum_{i=0}^{n-1}s^i\sum_{j=0}^{\lrf{(n-1-i)/2}}f_{n,i,j}^-(q)u^jv^{n-1-i-2j}.
\end{equation}
Applying the operator $D_{G_7}$ to both sides of the above expansion and setting $s=0$, we obtain
\begin{align*}
&D_{G_7}^{n+1}(J)\mid_{s=0}\\
&=D_{G_7}\left(D_{G_7}^{n}(J)\right)\mid_{s=0}\\
&=qH\sum_{j}f_{n,0,j}^+(q)u^jv^{n-2j}+(1+q)J\sum_{j}f_{n,1,j}^+(q)u^{j+1}v^{n-1-2j}+\\
&(1+q)J\sum_{j}jf_{n,0,j}^+(q)u^jv^{n+1-2j}+2(1+q)J\sum_{j}(n-2j)f_{n,0,j}^+(q)u^{j+1}v^{n-1-2j}+\\
&(qHv+Ju)\sum_{j=0}^{\lrf{(n-1)/2}}f_{n,0,j}^-(q)u^jv^{n-1-2j}+(1+q)H\sum_{j}f_{n,1,j}^-(q)u^{j+1}v^{n-2-2j}+\\
&(1+q)H\sum_{j}jf_{n,0,j}^-(q)u^jv^{n-2j}+2(1+q)H\sum_{j}(n-1-2j)f_{n,0,j}^-(q)u^{j+1}v^{n-2-2j}.
\end{align*}

Set $f_{n,j}^+(q)=f_{n,0,j}^+(q)$ and $f_{n,j}^-(q)=f_{n,0,j}^-(q)$.
Since $s$ marks fixed points, we have $$f_{n,1,j}^+(q)=nf_{n-1,0,j}^+(q)=nf_{n-1,j}^+(q),$$
$$f_{n,1,j}^-(q)=nf_{n-1,0,j}^-(q)=nf_{n-1,j}^-(q).$$
Comparing the coefficients of $Ju^jv^{n+1-2j}$ and $Hu^jv^{n-2i}$ on both sides of
$$D_{G_7}^{n+1}(J)\mid_{s=0}=D_{G_7}\left(D_{G_7}^{n}(J)\right)\mid_{s=0},$$
and simplifying yields
the following recurrence system:
\begin{align*}
f_{n+1,j}^+(q)&=(1+q)nf_{n-1,j-1}^+(q)+(1+q)jf_{n,j}^+(q)+2(1+q)(n-2j+2)f_{n,j-1}^+(q)+\\
&f_{n,j-1}^-(q),\\
f_{n+1,j}^-(q)&=qf_{n,j}^+(q)+qf_{n,j}^-(q)+(1+q)nf_{n-1,j-1}^-(q)+(1+q)jf_{n,j}^-(q)+\\
&2(1+q)(n-2j+1)f_{n,j-1}^-(q).
\end{align*}
Let
\begin{equation*}
f_{n}^{+}(x,q)=\sum_{j=0}^{\lrf{n/2}}f_{n,j}^+(q)x^j(1+x)^{n-2j},$$
$$f_n^-(x,q)=\sum_{j=0}^{\lrf{(n-1)/2}}f_{n,j}^-(q)x^j(1+x)^{n-1-2j}.
\end{equation*}
Comparing~\eqref{Dg6yx} and~\eqref{DG7uv}, and takeing $s=0,u=x$ and $v=1+x$ in~\eqref{DG7uv}, we immediately get~\eqref{fnbigamma}.
It is routine to deduce the recurrence system of the polynomials $f_{n}^{+}(x,q)$ and $f_n^-(x,q)$.
When $q\geq 0$, since the $\gamma$-coefficients $f_{n,i}^+(q)$ and $f_{n,i}^-(q)$ are both nonnegative,
so $f_{n}^{+}(x,q)$ and $f_{n}^{-}(x,q)$ are both $\gamma$-positive.
\end{proof}

\end{document}